\documentclass[11pt]{article}
\usepackage[pdftex]{graphicx}
\usepackage{amssymb,amsmath,latexsym, amsthm,amscd}
\usepackage{tikz} 
\usetikzlibrary{matrix,arrows}
\usepackage[margin=1.5in]{geometry} 
\usepackage[overload]{textcase}
\usepackage{hyperref}
  \hypersetup{
    colorlinks,
    citecolor=black,
    filecolor=black,
    linkcolor=black,
    urlcolor=black
}

\usepackage{graphicx}
\usepackage{mathptmx}
\usepackage{amssymb,amsmath,amsfonts,latexsym}
\usepackage{tikz} 
\newtheorem{thm}{Theorem}[section]
\newtheorem{conj}[thm]{Conjecture}
\newtheorem{defn}[thm]{Definition}
\newtheorem{cor}[thm]{Corollary}
\newtheorem{clm}[thm]{Claim}
\newtheorem{fact}[thm]{Fact}
\newtheorem{lem}[thm]{Lemma}

\begin{document}

\title{Artin Relations in the Mapping Class Group}
\author{JAMIL MORTADA}
\date{}

\maketitle

\begin{abstract}
For every integer $\ell \geq 2$, we find elements $x$ and $y$ in the mapping class group of an appropriate orientable surface $S$, satisfying the Artin relation of length $\ell$. That is, $xyx\cdots = yxy\cdots$, where each side of the equality contains $\ell$ terms. By direct computations, we first find elements $x$ and $y$ in Mod(S) satisfying Artin relations of every even length $\geq 8$, and every odd length $\geq 3$. Then using the theory of Artin groups, we give two more alternative ways for finding Artin relations in Mod(S). The first provides Artin relations of every length $\geq 3$, while the second produces Artin relations of every even length $\geq 6$. 
\end{abstract}

\section{Introduction}
\label{Intro}

Let $S = S_{g,b}$ be a compact orientable surface of genus $g$ with $b$ boundary components. Denote by Mod(S) the mapping class group of $S$, which is the group of isotopy classes of orientation preserving homeomorphisms of $S$ fixing $\partial S$ pointwise. In this paper, we investigate Artin relations in Mod(S). These relations are interesting because they allow one to relate Artin groups to mapping class groups. Finding arbitrary length Artin relations between mapping classes makes it easy to obtain representations (possibly faithful ones) of a large class of Artin groups into Mod(S). Indeed, knowing all the Artin relations satisfied by a finite collection of mapping classes, one can construct a Coxeter graph $\Gamma$ whose vertices are the mapping classes and whose edges (along with their labels) are determined by the lengths of the Artin relations between the mapping classes. In this case, there is a natural homomorphism from the Artin group $\mathcal{A}(\Gamma)$ to Mod(S). In this paper, we make use of Artin relations to find embeddings of the Artin  groups $A(I_2(\ell))$, $\ell \geq 3$, into Mod(S).
 \smallskip

If $\ell \geq 2$ is an integer and $a$ and $b$ are elements in a group $G$, we say that $a$ and $b$ satisfy the \textbf{Artin relation of length $\ell$} (or the $\ell$-Artin relation) if 
\begin{center}
$prod(a,b;\ell) = prod(b,a;\ell)$
\end{center}
\noindent where $prod(a,b;\ell) = \underbrace{aba\cdots}_{\ell}$
 
Given $\ell$ as above, we choose a suitable orientable surface $S$ (see below), and find elements $x$ and $y$ in Mod(S) such that $prod(x,y;\ell) = prod(y,x;\ell)$. This answers the second question on page 117 of \cite{W2}. We prove the following:

\begin{thm}\label{even Artin relation thm}
 Let $k \geq 2$ be an integer. Suppose $a_0,a_1,\cdots,a_k$ form a chain of simple closed curves in an orientable surface S. If $x = T_0$ and $y = T_1 \cdots T_k$, then
\begin{center}
 $prod(x,y;\ell) = prod(y,x;\ell) \Leftrightarrow \ell \equiv 0 \mathit{(mod(2k+4))}$
\end{center}
\end{thm}

\begin{figure}[t]
	\centering
		\includegraphics[width=0.75\textwidth]{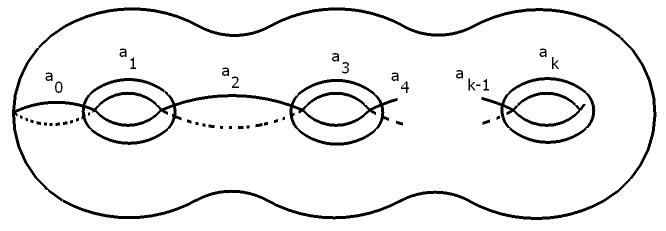}
	\caption{{\small Curves $a_0, \cdots, a_k$ form a chain of length $k+1$. If $x = T_0$ and $y = T_1 \cdots T_k$, then $prod(x,y;2k+4) = prod(y,x;2k+4)$.}}
	\label{Chain of length k+1}
\end{figure}
 
\noindent \textbf{Notation.} In Theorem~\ref{odd Artin relation theorem} below, we shall change the notation of a Dehn twist in order to make it easier for the reader to follow the proof. Instead of $T_i$, we shall denote Dehn twists by $A_i$ and $B_i$. $A_i$ and $B_i$ represent Dehn twists along curves $a_i$ and $b_i$ respectively.

\begin{thm} \label{odd Artin relation theorem}
Let $k$ be a positive integer, and suppose that $a_1,\cdots,a_k,b_1,\cdots,b_k$ form a curve chain $\mathcal{C}_{2k}$ in an orientable surface $S$. If $x = A_1 \cdots A_k$ and $y = B_1 \cdots B_k$, then
\begin{center}
 $prod(x,y;\ell) = prod(y,x;\ell) \Leftrightarrow \ell \equiv 0 \mathit{(mod(2k+1))}$
\end{center}
\end{thm}

In the next two theorems, $S_{A_i}$ and  $S_{D_i}$ represent surfaces canonically associated to the Coxeter graphs $A_i$ and $D_i$ shown below. 

\begin{thm} \label{bdrelthm1}
Let $k \geq 3$ be an integer. Suppose that $a_1,a_2,\cdots,a_{k-1}$ form a $(k-1)$-chain in $S_{A_{k-1}}$. Let
\[
 x = 
  \left\{
   \begin{array}{ll}
    T_1 T_3 \cdots T_{k-3} T_{k-1} & \mbox{when $k$ is even} \\
    T_1 T_3 \cdots T_{k-4} T_{k-2} & \mbox{when $k$ is odd}
   \end{array}
  \right.
\]
 \smallskip  
\[
 y = 
  \left\{
   \begin{array}{ll}
    T_2 T_4 \cdots T_{k-4} T_{k-2} & \mbox{when $k$ is even} \\
    T_2 T_4 \cdots T_{k-3} T_{k-1} & \mbox{when $k$ is odd}
   \end{array}
  \right.
\] 
\noindent Then $x$ and $y$ generate the Artin group $\mathcal{A}(I_2(k))$ in $Mod(S_{A_{k-1}})$. Moreover,
\begin{center}
 $prod(x,y;n) = prod(y,x;n)$ if and only if $n \equiv 0 \mathit{mod (k)}$
\end{center}
\end{thm}

\begin{thm} \label{bdrelthm2}
Let $ k \geq 4$ be an integer, and suppose that curves $a_1,a_2,\cdots,a_k$ have curve graph $D_k$ in $S_{D_k}$. Let
\[
 x = 
  \left\{
   \begin{array}{ll}
    T_1 T_3 \cdots T_{k-3} T_{k-1} T_k & \mbox{when $k$ is even} \\
    T_1 T_3 \cdots T_{k-2} & \mbox{when $k$ is odd} \\
    \end{array}
  \right.
\]
 \smallskip  
\[
 y = 
  \left\{
   \begin{array}{ll}
    T_2 T_4 \cdots T_{k-2} & \mbox{when $k$ is even} \\
    T_2 T_4 \cdots T_{k-3} T_{k-1} T_k & \mbox{when $k$ is odd}
   \end{array}
  \right.
\]
\noindent Then $x$ and $y$ generate the Artin group $\mathcal{A}(I_2(2k-2))$ in $Mod(S_{D_k})$. Moreover,
\begin{center}
 $prod(x,y;n) = prod(y,x;n)$ if and only if $n \equiv 0 \mathit{mod(2k-2)}$
\end{center}
\end{thm}

Two approaches are used to find Artin relations in Mod(S). The first approach includes some rather involved computations, which rely on the commutativity and braid relations between Dehn twists (Fact~\ref{comm reln} and~\ref{braid reln}). The second approach makes use of Artin groups, specifically the LCM-homomorphisms (definition~\ref{lcm hom}) induced by certain dihedral foldings. In this approach, we actually find elements $x$ and $y$ in Mod(S) that generate an Artin group with one relation. We briefly describe both methods for finding Artin relations.

\begin{figure}[t]
	\centering
		\includegraphics[width=0.82\textwidth]{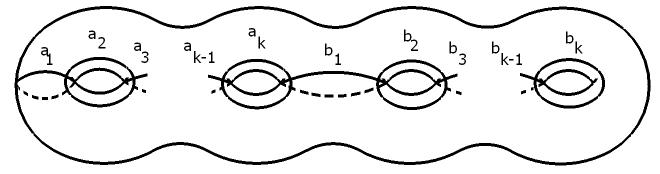}
	\caption{{\small Curves $a_1, \cdots, a_k, b_1, \cdots, b_k$ form a chain of length $2k$, $k \geq 1$. If $x = A_1 \cdots A_k$ and $y = B_1 \cdots B_k$, then $prod(x,y;2k+1) = prod(y,x;2k+1)$.}}
	\label{chain picture for odd Artin relation theorem}
\end{figure}

\begin{defn} \label{chain relation defn}  
  A finite collection $\{a_1,\cdots,a_p\}$ of pairwise non-isotopic simple closed curves in $S$ forms a chain of length $p$ ($p$-chain for short) if $i(a_j,a_{j+1}) = 1$ for all $j$ and $i(a_j,a_k) = 0$ for $|j-k| \geq 2$. Denote a $p-chain$ by $\mathcal{C}_p$.
\end{defn}

In the first approach, $x$ and $y$ are taken to be products of Dehn twists along curves in some curve chain $\mathcal{C}_p$. Depending on $\ell$, we choose $p$ accordingly. Given $p$, we choose an orientable surface $S$ with large enough genus to accommodate $\mathcal{C}_p$. Given $S$, we prescribe $x, y \in$ Mod(S) and prove that $prod(x,y;\ell) = prod(y,x;\ell)$. 
 \medskip
 
The second approach is rather different. Consider the Coxeter graphs $A_n$, $D_n$, and $I_2(n)$ (illustrated below), and their associated Artin groups $\mathcal{A}(A_n)$, $\mathcal{A}(D_n)$, and $\mathcal{A}(I_2(n))$. In this approach, we invoke the LCM-homomorphisms induced by the dihedral foldings $A_{n-1} \rightarrow I_2(n)$ and $D_n \rightarrow I_2(2n-2)$. The induced embeddings $\mathcal{A}(I_2(n)) \rightarrow \mathcal{A}(A_{n-1})$ and $\mathcal{A}(I_2(2n-2)) \rightarrow \mathcal{A}(D_n)$ provide $n$ and $(2n-2)$-Artin relations in $\mathcal{A}(A_{n-1})$ and $\mathcal{A}(D_n)$ respectively. By Theorem~\ref{pv} [Perron-Vannier], the geometric homomorphism (definition~\ref{geometric hom}) is actually a monomorphism for the Artin groups of types $A_n$ and $D_n$. As such, the Artin relations carry on, via the geometric homomorphism, to the corresponding mapping class groups. 
 \medskip
 
\begin{center}
\begin{tikzpicture}
\filldraw [black] (1.5,0) circle (2pt) 
(2.5,0) circle (2pt)
(3.5,0) circle (2pt)
(4.5,0) circle (2pt)
(5.5,0) circle (2pt)
(6.5,0) circle (2pt)
[black] (3.7,0) circle (0.25pt)
(3.7,0) circle (0.25pt)
(3.7,0) circle (0.25pt)
(4.1,0) circle (0.25pt)
(4.3,0) circle (0.25pt);
\draw (1.5,0) -- (2.5,0) -- (3.5,0);
\draw (4.5,0) -- (5.5,0) -- (6.5,0);
 \draw(1.5,-0.30)node {$s_1$};
 \draw(2.5,-0.30)node {$s_2$};
 \draw(3.5,-0.30)node {$s_3$};
 \draw(4.5,-0.30)node {$s_{n-2}$};
 \draw(5.5,-0.30)node {$s_{n-1}$};
 \draw(6.5,-0.30)node {$s_n$};
 \draw(0,0)node {$A_n =$};
 \draw(8.5,0)node {($n \geq 2)$};
\end{tikzpicture} \\
 \smallskip 
 
  \begin{tikzpicture}
\filldraw [black] (1.5,0) circle (2pt) 
(2.5,0) circle (2pt)
(3.5,0) circle (2pt)
(4.5,0) circle (2pt)
(5.5,0) circle (2pt)
(6.5,0.7) circle (2pt) 
(6.5,-0.7) circle (2pt)
[black] (3.7,0) circle (0.25pt)
(3.7,0) circle (0.25pt)
(3.9,0) circle (0.25pt)
(4.1,0) circle (0.25pt)
(4.3,0) circle (0.25pt);
\draw (1.5,0) -- (2.5,0) -- (3.5,0);
\draw (5.5,0) -- (6.5,0.7);
\draw (5.5,0) -- (6.5,-0.7);
\draw (4.5,0) -- (5.5,0);
 \draw(1.5,-0.30)node {$s_1$};
 \draw(2.5,-0.30)node {$s_2$};
 \draw(3.5,-0.30)node {$s_3$};
 \draw(4.3,-0.30)node {$s_{n-3}$};
 \draw(5.2,-0.30)node {$s_{n-2}$};
 \draw(6.5,1)node {$s_n$};
 \draw(6.5,-1)node {$s_{n-1}$};
 \draw(0,0)node {$D_n =$};
 \draw(8.5,0)node {($n \geq 4)$};
\end{tikzpicture}\\
 \smallskip 
 
 \begin{tikzpicture}
\filldraw [black] (1.5,0) circle (2pt) 
(2.5,0) circle (2pt);
\draw (1.5,0) -- (2.5,0);
 \draw(1.5,-0.30)node {$s$};
 \draw(2.5,-0.30)node {$t$};
 \draw(2,0.25)node{$n$};
 \draw(0.2,0)node {$I_2(n) =$};
 \draw(8.5,0)node {($n \geq 3)$};
\end{tikzpicture}\\
\smallskip
 
{\small The Coxeter graphs of types $A_n$, $D_n$, and $I_2(n)$}
\end{center}

\section{Artin groups}
\label{sec:1}
 
 \subsection{Basic facts}
  \label{sec:2}
  
  A Coxeter system of rank $n$ is a pair $(W,S)$ consisting of a finite set $S$ of order $n$ and a group $W$ with presentation
\begin{center}
 $\langle S \hspace{0.1cm} | \hspace{0.1cm}  s^2 = 1 \hspace{0.1cm} \forall \hspace{0.1cm} s \in S, \hspace{0.2cm} prod(s,t;m_{st}) = prod(t,s;m_{st}) \hspace{0.2cm} such \hspace{0.2cm} that \hspace{0.2cm} m_{st} \neq \infty \rangle$
\end{center}
\noindent where $m_{ss} = 1$ and $m_{st} = m_{ts} \in \{ 2,3,\cdots,\infty \}$ for $s \neq t$. $m_{st} = \infty$ means that there is no relation between $s$ and $t$. 
 \smallskip
 
A Coxeter system is determined by its Coxeter graph $\Gamma$. This graph has vertex set $S$ and includes an edge labeled $m_{st}$, between $s$ and $t$, whenever $m_{st} \geq 3$. The label $m_{st} = 3$ is usually omitted. The graph $\Gamma$ defines the type of a Coxeter group. We say that $W$ is a Coxeter group of type $\Gamma$, and denote it by $W(\Gamma)$. Alternatively, a Coxeter system can be uniquely determined by its Coxeter matrix $M = (m_{ij})_{i,j \in S}$, where $M$ is an $n \times n$ symmetric matrix with ones on the main diagonal and entries in $\{2, \cdots, \infty \}$ elsewhere. When $W$ is finite, we refer to it as a Coxeter group of finite type. Otherwise, $W$ is of infinite type. 
 \smallskip
 
The Artin group, $\mathcal{A}(\Gamma)$, of type $\Gamma$ has presentation
\begin{center}
  $\langle S \hspace{0.1cm} | \hspace{0.1cm} prod(s,t;m_{st}) = prod(t,s;m_{st}) \hspace{0.2cm} such \hspace{0.2cm} that \hspace{0.2cm} m_{st} \neq \infty \rangle$ ($\ast$)
\end{center}

It is clear from the presentations that $W(\Gamma)$ is a quotient of $\mathcal{A}(\Gamma)$. It is the quotient of $\mathcal{A}(\Gamma)$ by the normal closure of $\{ s^2 | s \in S \}$. We say that an Artin group has finite type, if its corresponding Coxeter group 
is finite. 
\smallskip

Consider $F(S)^+$, the free monoid (semigroup with $1$) of positive words in the alphabet of $S$. The Artin monoid $\mathcal{A}^{+}(\Gamma)$ of type $\Gamma$ is obtained from $F(S)^+$ by stipulating that $prod(s,t;m_{st}) \d{=} prod(s,t;m_{ts})$ for all $s,t \in S$ and $m_{st} \neq \infty$. The equality \d{=} denotes the positive word equivalence in $\mathcal{A}^{+}(\Gamma)$ (as opposed to the word equivalence in the group $\mathcal{A}(\Gamma)$ which is denoted by $=$). In other words, $\mathcal{A}^+(\Gamma)$ is given by ($\ast$), considered as a monoid presentation.

\begin{defn} \label{monoid homomorphism}
Let $M$ and $N$ be monoids. A map $\phi : M \rightarrow N$ is said to be a monoid homomorphism if $f(xy) = f(x)f(y)$ for all $x,y \in M$ and $f(1_M) = 1_N$. 
\end{defn}
\noindent We state the following useful facts about Artin monoids and Artin groups. More information can be found in \cite{BS}. 

\begin{itemize}
\item If $\Gamma$ is of finite type, then the canonical homomorphism $\mathcal{A}^{+}(\Gamma) \rightarrow \mathcal{A}(\Gamma)$ is injective.
\item The Artin monoid $\mathcal{A}^{+}(\Gamma)$ is cancellative. That is, $UA_1V$ \d{=} $UA_2V$ implies $A_1$ \d{=} $A_2$.
\item We say that $U$ divides $V$ (on the left), and write $U|V$, if $V$ \d{=} $UV'$ for some $V' \in \mathcal{A}^{+}(\Gamma)$. 
\item An element $V$ is said to be a common multiple for a finite subset $\mathcal{U} = \{ U_1,\cdots,U_r \}$ of $\mathcal{A}^{+}(\Gamma)$ if $U_i|V$ for each $i = 1,\cdots,r$. It is shown in [1] that if a common multiple of $\mathcal{U}$ exists, then there exists a necessarily unique least common multiple of $\mathcal{U}$. This least common multiple is a common multiple which divides all the common multiples of $\mathcal{U}$, and is denoted by $[U_1,\cdots,U_r]$. For each pair of element $s,t \in S$ with $m_{st} \neq \infty$, the least common multiple $[s,t]$ \d{=} $prod(s,t;m_{st})$. If $m_{st} = \infty$, then $s$ and $t$ have no common multiple.
\end{itemize}

 \subsection{LCM-homomorphisms and dihedral foldings}
  \label{sec:3}
  
  The majority of definitions and results from this subsection are due to Crisp. See \cite{Cr} for more details. 

\begin{defn} \label{def lcm}
An Artin monoid homomorphism $\phi : \mathcal{A}^+(\Gamma) \rightarrow \mathcal{A}^+(\Gamma')$ respects lcms if
\begin{enumerate}
 \item $\phi(s) \neq 1$ for each generator $s$, and 
 \item For each pair of generators $s,t \in S$, the pair $\phi(s), \phi(t)$ have a common multiple only if $s$ and $t$ do. In that case,       $[\phi(s),\phi(t)]$ \d{=} $\phi([s,t])$.
\end{enumerate}
\end{defn}

\begin{thm}[Crisp] \label{cr1}
A homomorphism $\phi : \mathcal{A}^+(\Gamma) \rightarrow \mathcal{A}^+(\Gamma')$ between Artin monoids which respects lcms is injective.
\end{thm}

\begin{thm}[Crisp] \label{cr2}
If $\phi : \mathcal{A}^+(\Gamma) \rightarrow \mathcal{A}^+(\Gamma')$ is a monomorphism between finite type Artin monoids, then the induced homomorphism $\phi_A : \mathcal{A}(\Gamma) \rightarrow \mathcal{A}(\Gamma')$ between Artin groups is injective. 
\end{thm}
 
Let $\mathcal{A}^+(\Gamma)$ be an Artin monoid with generating set $S$. If $T \subseteq S$ has a common multiple, we denote its least common multiple by $\Delta_T$. $\Delta_T$ is also called the fundamental element for $T$. It is shown in \cite{BS} that $\Delta_T$ exists if and only if the parabolic subgroup $W_T$ (ie the subgroup of $W$ generated by $T$) is finite.

\begin{defn} \label{lcm hom} 
Let $(W,S)$ and $(W',S')$ be Coxeter systems of types $\Gamma$ and $\Gamma'$ respectively, and assume $m_{st} \neq \infty$ for all $s,t \in S$. Let $\{T(s) | s \in S \}$ be a collection of mutually disjoint subsets of $S'$ such that
\begin{enumerate}
 \item for each $s \in S$, $T(s)$ is nonempty and $\Delta_{T(s)}$ exists, and 
 \item $prod(\Delta_{T(s)},\Delta_{T(t)};m_{st})$ \d{=} $prod(\Delta_{T(t)},\Delta_{T(s)};m_{ts})$ \d{=} $[\Delta_{T(s)},\Delta_{T(t)}]$ for all $s,t \in S$.
\end{enumerate}
\noindent Define a homomorphism $\phi_T : \mathcal{A}^+(\Gamma) \rightarrow \mathcal{A}^+(\Gamma')$ by $\phi_T(s)$ \d{=} $\Delta_{T(s)}$ for $s \in S$. Such a homomorphism is called an LCM-homomorphism.
\end{defn}
 \medskip
\noindent It is clear from condition $2$ that $\phi_T$ is a homomorphism (For each relation $R$ in $\mathcal{A}^+(\Gamma)$, $\phi_T(R)$ is a relation in $\mathcal{A}^+(\Gamma')$). Additionally, $\phi_T$ respects lcms. Indeed, condition 1 of Definition \ref{def lcm} is satisfied because $T(s) \neq \emptyset$ consists of generators of $S'$. As such $\Delta_{T(s)} \neq 1$. Moreover, the assumption $m_{st} \neq \infty$ for all $s,t$ guarantees the existence of $[s,t]$ for all $s,t$. Also, condition 2 of Definition \ref{lcm hom} implies the second condition of Definition \ref{def lcm}. Since LCM-homomorphisms respect lcms, they are injective by Theorem \ref{cr1}.  
 \smallskip
 
Let $(W,S)$ be an irreducible Coxeter system (ie its Coxeter graph $\Gamma$ is connected), with $S = \{ s_1,s_2,\cdots,s_n \}$. A \textbf{Coxeter element} $h$ of $W$ is defined to be a product $s_{\sigma(1)} s_{\sigma(2)} \cdots s_{\sigma(n)}$, where $\sigma \in \Sigma_n$. It is known that all Coxeter elements are conjugate in $W$ (See p.74 in \cite{H}). Hence, all the Coxeter elements have the same order in $W$. Consequently, the Coxeter number of $W$ is defined to be the order of a Coxeter element. It is well known that the Coxeter graphs $A_n$, $D_n$, and $I_2(n)$ have Coxeter numbers $n+1$, $2n-2$, and $n$ respectively (See \cite{H}).
 
\begin{defn} \label{dihedral folding}
Let $\Gamma$ and $\Gamma'$ be Coxeter graphs with respective vertex sets $S$ and $S'$. A dihedral folding of $\Gamma'$ onto $\Gamma$ is a surjective simplicial map $f : \Gamma' \rightarrow \Gamma$ such that for every edge $\epsilon$ between $s$ and $t$ (labeled $m > 3$) in $\Gamma$, the restriction $f_{\epsilon}$ of $f$ to $f^{-1}(\epsilon)$ is described as follows:
 \medskip

\noindent The preimage $f^{-1}(\epsilon)$ is an irreducible finite type Coxeter graph $K$ with Coxeter number $m > 3$ (Note that this $m$ is the same as the one above). Choose a partition $K_s \cup K_t$ of the vertex set of $K$ so that there are no edges between the vertices in $K_s$ and no edges between the vertices in $K_t$. Then the dihedral folding $f_{\epsilon} : K \rightarrow \epsilon$ is the unique simplicial map such that $f(K_s) = s$ and $f(K_t) = t$. One may always choose such a partition, and it is unique up to relabeling of the two sets.
\end{defn}

Since $f_{\epsilon}$ above, depends on the choice of labeling the partition of $K$ into $K_s \cup K_t$, there are possibly two distinct foldings of $\Gamma'$ onto $\Gamma$. To distinguish them, these foldings are denoted $(K,+\epsilon)$ and $(K,-\epsilon)$. 

\begin{thm}[Crisp] \label{cr3} A dihedral folding $f : \Gamma' \rightarrow \Gamma$ induces an LCM-homomorphism $\phi^f : \mathcal{A}^+(\Gamma) \rightarrow \mathcal{A}^+(\Gamma')$ defined by $\phi^f(s)$ \d{=} $\Delta_{f^{-1}(s)}$ for $s \in S$.
\end{thm}

\begin{cor}\label{corollary to cr3}
Let $\Gamma(h)$ be the Coxeter graph corresponding to an irreducible finite type Coxeter group with Coxeter number $h$. Then the dihedral folding of $\Gamma(h)$ onto $I_2(h)$ defines an embedding $\mathcal{A}^+(I_2(h)) \rightarrow \mathcal{A}^+(\Gamma(h))$ between the Artin monoids. By Theorem~\ref{cr2}, there is an embedding between the corresponding Artin groups. 
\end{cor} 

The following lemma (Reduction Lemma) is due to Brieskorn and Saito. It is proved in \cite{BS}.

\begin{lem}[Reduction Lemma] \label{red lem} 
Let $(W,S)$ be a Coxeter system with Coxeter graph $\Gamma$. If $X,Y \in \mathcal{A}^+(\Gamma)$ and $s,t \in S$ satisfy $sX$ \d{=} $tY$, then $\exists \hspace{0.1cm} W \in \mathcal{A}^+(\Gamma)$ such that
\begin{center}
$X$ \d{=} $prod(t,s;m_{st}-1)W$ and $Y$ \d{=} $prod(s,t;m_{st}-1)W$
\end{center}
\end{lem}
 
\begin{lem} \label{lcm lem}
Let $(W,S)$ be a finite type Coxeter system with Coxeter graph $\Gamma$. If $T = \{ t_1,\cdots,t_k \}$ is a subset of $S$ consisting of pairwise commuting generators, then the least common multiple ,$\Delta_T$, of $T$ in $\mathcal{A}^+(\Gamma)$ exists and is given by $t_{\sigma (1)} t_{\sigma (2)} \cdots t_{\sigma (k)}$, $\sigma \in \Sigma_k$
\end{lem}
\begin{proof}
Set $\alpha$ \d{=} $t_1 t_2 \cdots t_k$. Since the $t_i$ pairwise commute, $\alpha$ is a common multiple of $T$. Suppose that $\beta$ is another common multiple of $T$. Then for each $i = 1,\cdots,k$, $\exists \hspace{0.1cm} x_i \in \mathcal{A}^+(\Gamma)$ such that $\beta$ \d{=} $t_i x_i$. In particular, $t_1x_1$ \d{=} $t_{j_1}x_{j_1}$ for all $j_1 \in \{ 2,\cdots,k \}$. By Lemma~\ref{red lem}, $\exists \hspace{0.1cm} W_{1j_1} \in \mathcal{A}^+(\Gamma)$ such that $x_1$ \d{=} $t_{j_1}W_{1j_1}$ for all $j_1$ (note that we used the assumption $m_{t_1 t_{j_1}} = 2$). In particular, $t_2W_{12}$ \d{=} $t_{j_2}W_{1j_2}$ for all $j_2 \in \{ 3,\cdots,k \}$. By Lemma~\ref{red lem}, $\exists \hspace{0.1cm} W_{12j_2} \in \mathcal{A}^+(\Gamma)$ such that $W_{12}$ \d{=} $t_{j_2}W_{12j_2}$ for all $j_2$. In particular, $t_3W_{123}$ \d{=} $t_{j_3}W_{12j_3}$ for all $j_3 \in \{ 4,\cdots,k \}$. By Lemma~\ref{red lem}, $\exists \hspace{0.1cm} W_{123j_3} \in \mathcal{A}^+(\Gamma)$ such that $W_{123}$ \d{=} $t_{j_3}W_{123j_3}$ for all $j_3$. By repeating the same process, one gets $W_{1 2 \cdots r}$ \d{=} $t_{j_r}W_{1 2 \cdots j_r}$ for all $j_r \in \{r+1,\cdots,k \}$, where $r \in \{ 4,\cdots,k \}$. In particular, when $j_2 = 3$, $j_3 = 4$ and $j_r = r+1$ for all $r \in \{ 4,\cdots,k \}$, the following hold equalities hold in 
$\mathcal{A}^+(\Gamma)$:
\begin{eqnarray*}
W_{12} &\d{=}& t_3W_{123} \\
W_{123}&\d{=}& t_4W_{1234} \\
&\vdots& \\ 
W_{1 2 \cdots k-1} &\d{=}& t_{k}W_{1 2 \cdots k}
\end{eqnarray*}
\noindent Hence, $\beta$ \d{=} $t_1 x_1$ \d{=}  $t_1t_2W_{12}$ \d{=} $t_1t_2t_3W_{123}$ \d{=} $\cdots$ \d{=} $t_1 t_2 \cdots t_k W_{1 2 \cdots k}$  \d{=} $\alpha W_{1 2 \cdots k}$. Therefore, $\alpha | \beta$ and $\Delta_T$ \d{=} $\alpha$.
\end{proof}

\section{Geometric homomorphisms}
\label{sec:4}

\begin{fact} \label{comm reln}
Suppose that $a_1$ and $a_2$ are isotopy classes of simple closed curves in an orientable surface $S$. If $i(a_1,a_2) = 0$, then $[T_1,T_2] = 1$. This is called the commutativity or disjointness relation.
\end{fact}

\begin{fact} \label{braid reln}
Suppose that $a_1$ and $a_2$ are isotopy classes of simple closed curves in an orientable surface $S$. If $i(a_1,a_2) = 1$, then $T_1T_2T_1 = T_2T_1T_2$. This is called the braid relation.
\end{fact}

In Facts~\ref{comm reln} and~\ref{braid reln}, $T_1$ and $T_2$ denote Dehn twists along $a_1$ and $a_2$ respectively, and $i(a_1,a_2)$ represents the geometric intersection number between $a_1$ and $a_2$. Note that in the context of this paper, the braid relation of Fact~\ref{braid reln} will be referred to as a Artin relation of length $3$, or a $3$-Artin relation.
   
\begin{defn}\label{curve graph}
Let $\mathcal{C} = \{a_1, \cdots, a_n\}$ be a finite collection of pairwise nonisotopic simple closed curves in $S$. Assume that no two elements in $\mathcal{C}$ cobound a bigon and no three elements have a common point of intersection. To this collection, we associate a graph $\mathcal{CG}$, called the \textbf{curve graph} of the $a_i$, whose vertices are the curves $a_i$. Moreover, two vertices $a_i$ and $a_j$ are joined by an edge, labeled $i(a_i,a_j)$, whenever $i(a_i,a_j) > 0$. If $i(a_i,a_j) = 1$, suppress the label. 
\end{defn}
  
Let $\mathcal{A}(\Gamma)$ be an Artin group of small type. That is, $m_{ij} \leq 3$ for all $i,j$. Let $a_1, \cdots, a_n$ be a collection of simple closed curves in $S$ whose curve graph is isomorphic to $\Gamma$. Since $\Gamma$ is of small type, no two curves in the collection intersect more than once. There is a natural homomorphism $\mathcal{A}(\Gamma) \rightarrow Mod(S)$ mapping the $i^{th}$ generator $\sigma_i$ of $\mathcal{A}(\Gamma)$ to the Dehn twist $T_i$ along $a_i$. That this map is a homomorphism follows immediately from Facts ~\ref{comm reln} and ~\ref{braid reln} and the definitions of Coxeter and curve graphs.

\begin{defn}\label{geometric hom}
A homomorphism $\mathcal{A}(\Gamma) \rightarrow Mod(S)$ is said to be geometric if it maps the standard generators of $\mathcal{A}(\Gamma)$ to Dehn twists in Mod(S). 
\end{defn}
 \medskip
 
\noindent \textbf{Notation} Let $\Gamma$ be a small type Coxeter graph which is a tree. Let $S$ be an arbitrary surface and $a_1,\cdots,a_l$ be simple closed curves in $S$ with $i(a_j,a_k) \in \{0,1\}$, so that the curve graph $\mathcal{CG}$ associated with  $\{a_i\}_{i = 1}^l$ is isomorphic to $\Gamma$. Let $S_{\Gamma}$ be the closure of a regular neighborhood of $\cup_{i=1}^{l}a_i$ in $S$. It is very simple to check that the homeomorphism type of $S_{\Gamma}$ is independent of $S$ or $\{a_i\}$, and only depends on $\Gamma$, justifying the notation $S_{\Gamma}$.

\begin{thm}[Perron-Vannier] \label{pv}
If a curve graph $\mathcal{CG}$ is isomorphic to a Coxeter graph $\Gamma$ of type $A_n$ or $D_n$, then the geometric homomorphism $g:\mathcal{A}(\Gamma) \rightarrow Mod(S_{\Gamma})$ is injective.
\end{thm}

\begin{fact} \label{San thm}
Suppose $n \geq 3$ be an integer. If curves $a_1, a_2, \cdots, a_n$ form a $n$-chain in some surface, then $S_{A_n}$ is homeomorphic to $S_{\frac{n}{2},1}$ when $n$ is even, and $S_{A_n}$ is homeomorphic to $S_{\frac{n-1}{2},2}$ when $n$ is odd.
\end{fact}

\begin{fact} \label{Sdn thm}
Suppose $n \geq 4$ be an integer. If curves $a_1, a_2, \cdots, a_n$ have curve graph $D_n$ in some surface, then $S_{D_n}$ is homeomorphic to $S_{\frac{n-2}{2},3}$ when $n$ is even, and $S_{D_n}$ is homeomorphic to $S_{\frac{n-1}{2},2}$ when $n$ is odd.
\end{fact} 

\section{Artin relations from direct computations}
   \label{sec:5}
   
   \subsection{Artin relations of even length}
    \label{sec:6}
   
In this subsection we prove Theorem~\ref{even Artin relation thm}, which involves finding Artin relations of even lengths in Mod(S). Let $k \geq 2$ be an integer. For every integer multiple $\ell$ of $2k+4$, Theorem~\ref{even Artin relation thm} gives explicit elements $x$ and $y$ in the mapping class group of some appropriate orientable surface, so that $prod(x,y;\ell) = prod(y,x;\ell)$. By  \verb'"'appropriate\verb'"' orientable surface, we mean one with large enough genus to accommodate a chain $\mathcal{C}_{k+1}$ of $k+1$ curves. 
 \medskip

\noindent \textbf{Remark.} When $k = 1$, $x = T_0$ and $y = T_1$. Since $i(a_0,a_1) = 1$, it follows from Fact~\ref{braid reln} that $xyx = yxy$. This implies that $prod(x,y;\ell) = prod(y,x;\ell)$ for all integers $\ell$ which are multiples of $6$. As such, the sufficient condition of Theorem~\ref{even Artin relation thm} is true. However, the necessary condition does not hold because $xyx = yxy$ while $\ell \equiv 3 \mathit{(mod6)}$.
 \medskip
 
We proceed with the proof of Theorem~\ref{even Artin relation thm} as follows. First, we prove that
\begin{center}
$prod(x,y;\ell) \stackrel{(\ast)}{=} prod(y,x;\ell)$
\end{center}
\smallskip

\noindent does not hold in Mod(S) when $\ell \in \{2,3,4,5,6,7\}$. Next, we show that $(\ast)$ holds in Mod(S) when $k = 2$ and $\ell = 8$, and that $(\ast)$ is not true when $k > 2$ and $\ell = 8$. Then, we prove two claims (Claim~\ref{claim1e} and~\ref{claim2e}) which generalize the computations for $\ell \leq 8$ and allow us to show Theorem~\ref{even Artin relation thm} in general.
\smallskip

To check that $(\ast)$ does not hold in Mod(S) when $\ell \in \{2,3,4,5,6,7\}$, we proceed as follows. For $\ell \in \{2,3,4,5,6\}$ (starting with $\ell = 2$), replace $x$ with $T_0$ and $y$ with $T_1 \cdots T_k$ in $(\ast)$. If the leftmost Dehn twists on each side are equal, we cancel them to simplify the expression. Moreover, whenever possible, we apply Facts~\ref{comm reln} and~\ref{braid reln} to rearrange Dehn twists so that the leftmost ones on each side are identical. We cancel identical leftmost Dehn twists (on each side) as much as possible to obtain a simplified expression which is equivalent to $(\ast)$. Call this expression $E_{\ell}$ and assume it is given by $P_{\ell} = Q_{\ell}$. By showing that $E_{\ell}$ does not hold in Mod(S), this implies that $(\ast)$ is not true in Mod(S) (for that particular $\ell$). Now, multiply each side of $E_{\ell}$ by $x = T_0$ or $y = T_1 \cdots T_k$ on the right accordingly (ie $P_{\ell} x = Q_{\ell} y$ or $P_{\ell} y = Q_{\ell} x$) to obtain an equivalent expression for 
\begin{center}
$prod(x,y;\ell + 1) \stackrel{(\ast \ast)}{=} prod(y,x;\ell + 1)$
\end{center}
\noindent Rearrange and cancel identical Dehn twists on the left to obtain a simplified expression $E_{\ell + 1}$ which is equivalent to $(\ast \ast)$. Now, use $E_{\ell + 1}$ to show that $(\ast \ast)$ does not hold in Mod(S). Finally, we multiply $P_7$ on the right by $y$ and $Q_7$ on the right by $x$ in the expression $E_7 : P_7 = Q_7$. This gives $E_8 : P_7 y = Q_7 x$, which is equivalent to 
\begin{center}
$prod(x,y;8) = prod(y,x;8)$
\end{center}
\noindent Using Facts~\ref{comm reln} and~\ref{braid reln} to rearrange the Dehn twists in $P_7 y$ and $Q_7 x$, we show that the expression $E_8$ is true in Mod(S) when $k = 2$ and false when $k > 2$. 
 \smallskip
 
We will now prove that $x$ and $y$ do not satisfy an Artin relation of length $\ell \leq 7$. In the computations below, we shall only cancel on the left. Right cancellations are intentionally ignored. This simplifies things, as it allows us to start a new computation by using the result from the previous one. Throughout, we shall use Facts ~\ref{comm reln} and ~\ref{braid reln} to rearrange Dehn twists so that the leftmost ones on each side are identical.
\begin{eqnarray*}
 xy = yx \Leftrightarrow T_0T_1 \cdots T_k &=& T_1T_2 \cdots T_kT_0\\
         \Leftrightarrow T_0T_1 \cdots T_k &=& T_1T_0T_2 \cdots T_k
\end{eqnarray*}
\noindent Since $i(a_0,a_1) = 1$, $T_0T_1 \neq T_1T_0$. As such, the last equality on the right hand side (RHS) does not hold.
 \smallskip
 
We now describe a method that will be used throughout this article without further explicit mention. The equation $xy = yx$ above does not hold. However, the algebraic manipulations for the equivalence of $xy = yx$ with the last equation above do hold. Next, we multiply these equations by $x$ and $y$ on the right accordingly. The computations above imply:
\begin{eqnarray*}
 xyx = yxy \Leftrightarrow (T_0T_1 \cdots T_k)T_0 &=& T_1T_0T_2 \cdots T_k(T_1 \cdots T_k)\\
           \Leftrightarrow (T_0T_1T_0)T_2 \cdots T_k &=& T_1T_0T_2T_1T_3T_2T_4T_3\cdots T_iT_{i-1}T_{i+1}T_i \cdots\\
           && T_{k-1}T_{k-2}(T_kT_{k-1}T_k)\\
           \Leftrightarrow T_1T_0T_1T_2 \cdots T_k &=&  T_1T_0T_2T_1T_3T_2T_4T_3\cdots T_iT_{i-1}T_{i+1}T_i \cdots \\
           &&(T_{k-1}T_{k-2}T_{k-1})T_kT_{k-1}\\
           \Leftrightarrow T_1T_0T_1T_2 \cdots T_k &=& T_1T_0T_1T_2T_1T_3T_2T_4T_3\cdots T_iT_{i-1}T_{i+1}T_i \cdots \\
           &&T_{k-3}T_{k-1}T_{k-2}T_kT_{k-1}
\end{eqnarray*}
\begin{eqnarray*}
           \Leftrightarrow T_1T_0T_1T_2 \cdots T_k &=& T_1T_0T_1T_2T_3T_4 \cdots T_kT_1T_2T_3 \cdots T_{k-2}T_{k-1}\\
           \Leftrightarrow 1 &=& T_1 \cdots T_{k-1} 
\end{eqnarray*}
\noindent Since $T_1 \cdots T_{k-1}(a_2) = a_3$ or $T_1(a_2)$, the last equation of RHS does not hold.
 \smallskip
 
Using the above equivalences of $xyx = yxy$, and only left cancellation, we do the following for $(xy)^2 = (yx)^2$.
\begin{eqnarray*}
 (xy)^2 = (yx)^2 \Leftrightarrow T_1 \cdots T_{k-1}T_k &=& T_1 \cdots T_{k-1}T_0\\
                 \Leftrightarrow                   T_k &=& T_0  
\end{eqnarray*}
\noindent Since $k \geq 2$ by assumption, the last equality of RHS is obviously not true. 
\begin{eqnarray*}
 (xy)^2x = (yx)^2y \Leftrightarrow T_kT_0 &=& T_0T_1 \cdots T_k\\
                   \Leftrightarrow T_0T_k &=& T_0T_1 \cdots T_k\\
                   \Leftrightarrow    T_k &=& T_1 \cdots T_k
\end{eqnarray*}
\noindent Since $1 \neq T_1 \cdots T_{k-1}$ (see above), the last equality of RHS does not hold.
\begin{eqnarray*}
 (xy)^3 = (yx)^3 \Leftrightarrow T_kT_1 \cdots  T_{k-2}T_{k-1}T_k    &=& T_1 \cdots T_kT_0\\
                 \Leftrightarrow T_1 \cdots T_{k-2}(T_kT_{k-1}T_k)   &=& T_1 \cdots T_kT_0\\
                 \Leftrightarrow T_1 \cdots T_{k-2}T_{k-1}T_kT_{k-1} &=& T_1 \cdots T_kT_0\\
                 \Leftrightarrow                             T_{k-1} &=& T_0  
\end{eqnarray*} 
\noindent Since $k \geq 2$ by assumption, $T_{k-1} \neq T_0$.    
\begin{eqnarray*}
 (xy)^3x = (yx)^3y \Leftrightarrow T_{k-1}T_0 &=& T_0T_1 \cdots T_k
\end{eqnarray*} 
\noindent Since $i(a_{k-1},a_k) = 1$, $T_{k-1}T_k(a_{k-1}) = a_k$ and so $T_0T_1 \cdots T_k(a_{k-1}) = a_k$ for all $k \geq 2$. On the other hand, $T_{k-1}T_0(a_{k-1})$ equals $a_0$ when $k = 2$ and $a_{k-1}$ when $k > 2$. As such, $T_{k-1}T_0 \neq T_0T_1 \cdots T_k$.
\begin{eqnarray*}
 (xy)^4 = (yx)^4 \Leftrightarrow T_{k-1}T_1 \cdots T_{k-2}T_{k-1}T_k          &=& T_1 \cdots T_kT_0\\
                 \Leftrightarrow T_1 \cdots T_{k-3}(T_{k-1}T_{k-2}T_{k-1})T_k &=& T_1 \cdots T_kT_0\\
                 \Leftrightarrow T_1 \cdots T_{k-3}T_{k-2}T_{k-1}T_{k-2}T_k   &=& T_1 \cdots T_kT_0\\
                 \Leftrightarrow T_1 \cdots T_{k-3}T_{k-2}T_{k-1}T_kT_{k-2}   &=& T_1 \cdots T_kT_0\\
                 \Leftrightarrow                                      T_{k-2} &=& T_0 
\end{eqnarray*}
\noindent which is true when true when $k = 2$ and false when $k > 2$. This shows that $(\ast)$ holds in Mod(S) for $k = 2$ and $\ell = 8$ and is not true when $k > 2$ and $\ell = 8$.
\smallskip
 
To prove Theorem~\ref{even Artin relation thm} in general, we make the following claims:

\begin{clm} \label{claim1e}
Let $k$ and $i$ be positive integers such that $k \geq 2$ and $3 \leq i \leq k+1$. Then, for all $i$, 
\begin{center}
$prod(x,y;2i-1) = prod(y,x;2i-1) \Leftrightarrow T_{k-i+3} = T_1 \cdots T_k$
\end{center} 
\end{clm}

\begin{clm} \label{claim2e}
Let $k$ and $i$ be positive integers such that $k \geq 2$ and $3 \leq i \leq k+1$. Then, for all $i$,
\begin{center}
$prod(x,y;2i) = prod(y,x;2i) \Leftrightarrow T_{k-i+2} = T_0$
\end{center}
\end{clm}

\begin{proof}[Proof of Claim~\ref{claim1e} and~\ref{claim2e}]
To prove Claim~\ref{claim1e}, we proceed by induction on $i$. The base case, $i = 3$, has been proven above. Assume, by induction, that Claim~\ref{claim1e} holds for some $i \in \{3,\cdots,k\}$. We would like to show Claim~\ref{claim1e} holds for $i+1$. That is, we need to prove:
\begin{eqnarray*}
 prod(x,y;2i+1) = prod(y,x;2i+1) \Leftrightarrow T_{k-(i+1)+3} &=& T_1 \cdots T_k \\
                                 \Leftrightarrow T_{k-i+2} &=& T_1 \cdots T_k
\end{eqnarray*}

Assuming Claim~\ref{claim1e} for $i$ implies that $prod(x,y;2i) = prod(y,x;2i) \Leftrightarrow$
\begin{eqnarray*}
                    T_{k-i+3}T_1 \cdots T_{k-i+2}T_{k-i+3} \cdots T_k &=& T_1 \cdots T_{k-i+2} \cdots T_kT_0 \\
  \Leftrightarrow T_1 \cdots (T_{k-i+3}T_{k-i+2}T_{k-i+3}) \cdots T_k &=& T_1 \cdots T_{k-i+2} \cdots T_kT_0 \\
               \Leftrightarrow T_{k-i+2}T_{k-i+3}T_{k-i+2}T_{k-i+4} \cdots T_k &=& T_{k-i+2} \cdots T_kT_0 \\
                                 \Leftrightarrow T_{k-i+2} \cdots T_kT_{k-i+2} &=& T_{k-i+2} \cdots T_kT_0 \\
                                                                     T_{k-i+2} &=& T_0 \\
\end{eqnarray*}
\noindent To justify the above calculation, note that the $i$ under consideration belongs to $\{3,\cdots,k\}$. Since $k \geq 2$, it follows that $k-i+3 \in \{3,\cdots,k\}$. As such, $[T_{k-i+3},T_1] = 1$. In particular, this shows that Claim~\ref{claim1e} for some positive integer $i$, $3 \leq i \leq k+1$ and $k \geq 2$, implies Claim~\ref{claim2e} for that $i$. Given the equivalence $prod(x,y;2i) = prod(y,x;2i) \Leftrightarrow T_{k-i+2} = T_0$, then
\begin{eqnarray*}
  prod(x,y;2i+1) = prod(y,x;2i+1) \Leftrightarrow T_{k-i+2}T_0 &=& T_0T_1 \cdots T_k\\
                                  \Leftrightarrow T_0T_{k-i+2} &=& T_0T_1 \cdots T_k\\
                                      \Leftrightarrow T_{k-i+2}&=& T_1 \cdots T_k 
\end{eqnarray*}
\noindent The above calculation is justified because $k-i+2 \in \{2,\cdots,k-1\}$, and so $[T_{k-i+2},T_0] = 1$. This concludes the proof of  Claim~\ref{claim1e}. 
 \smallskip

By the above remark, the proof of Claim~\ref{claim2e} follows immediately from Claim~\ref{claim1e} and its proof. 
\end{proof}

\begin{proof}[Proof of Theorem~\ref{even Artin relation thm}]
Assume $\ell \equiv 0 \mathit{(mod(2k+4))}$. Since
\begin{center}
 $prod(x,y;2k+4) = prod(y,x;2k+4) \Rightarrow prod(x,y;q(2k+4)) = prod(y,x;q(2k+4))$
\end{center}
\noindent for all positive integers $q$, it suffices to show $prod(x,y;2k+4) = prod(y,x;2k+4)$. By Claim~\ref{claim2e}, we have:
\begin{eqnarray*}
   prod(x,y;2k+2) &=& prod(y,x;2k+2)\\ 
  \Leftrightarrow T_{k-(k+1)+2} &=& T_0 \\
  \Leftrightarrow           T_1 &=& T_0
\end{eqnarray*} 
\noindent which is not true. Given this, then
\begin{eqnarray*}
 prod(x,y;2k+3) = prod(y,x;2k+3) \Leftrightarrow T_1T_0 &=& T_0T_1 \cdots T_k 
\end{eqnarray*} 
\noindent Since $T_1T_0(a_1) = a_0 \neq a_2 = T_0T_1 \cdots T_k(a_1)$, $T_1T_0 \neq T_0T_1 \cdots T_k $. Finally,
\begin{eqnarray*}
 prod(x,y;2k+4) = prod(y,x;2k+4) \Leftrightarrow (T_1T_0T_1) \cdots T_k &=& T_0T_1 \cdots T_kT_0 \\ 
                                   \Leftrightarrow T_0T_1T_0 \cdots T_k &=& T_0T_1T_0  \cdots T_k
\end{eqnarray*}
\noindent which is true. 
 \smallskip
 
Conversely, assume $\ell$ is not a multiple of $2k+4$. Then
\begin{center}
 $prod(x,y;\ell) = prod(y,x;\ell) \Leftrightarrow prod(x,y;r) = prod(y,x;r)$
\end{center}
\noindent for some $r \in \left\{1, \cdots, 2k+3 \right\}$, where $\ell \equiv r \mathit{(mod(2k+4))}$.
 \smallskip

If $r = 2k+3$, it was shown above that $prod(x,y;r) \neq prod(y,x;r)$. 
 \smallskip
 
If $r < 2k+3$ is odd; say $r = 2s-1$ for some positive integer $s$, then
\begin{center}
 $prod(x,y;r) = prod(y,x;r) \Leftrightarrow T_{k-s+3} = T_1 \cdots T_k$
\end{center}
\noindent by Claim~\ref{claim1e}. But then $T_{k-s+3}(a_{k-s+3}) = a_{k-s+3}$ while
\begin{eqnarray*}
 T_1 \cdots T_k(a_{k-s+3}) &=& T_1 \cdots T_{k-s+2}T_{k-s+3}T_{k-s+4}(a_{k-s+3})\\
                           &=& T_1 \cdots T_{k-s+2}(a_{k-s+4})\\
                           &=& a_{k-s+4} \neq a_{k-s+3}
\end{eqnarray*}
If $r$ is even; say $r = 2s$, then by Claim~\ref{claim2e}, 
\begin{center}
$prod(x,y;r) = prod(y,x;r) \Leftrightarrow T_{k-s+2} = T_0 \Leftrightarrow k-s+2 = 0 \Leftrightarrow r = 2k+4$
\end{center}
\end{proof}

\begin{conj}
Let $a_0,a_1,\cdots,a_k$ and $T_i$, $i \in \left\{0,\cdots,k\right\}$ be as in Theorem~\ref{even Artin relation thm}. Let $x = T_0$ and 
$y = T_{\sigma(1)}T_{\sigma(2)} \cdots T_{\sigma(k)}$, where $\sigma \in S_k$. Then
\begin{center} 
  $prod(x,y;n) = prod(y,x;n) \Leftrightarrow n \equiv 0 \mathit{(mod(2k+4))}$
\end{center}
\end{conj}

\noindent The conjecture holds when $k = 2,3,$ and $4$. This has been proven by brute force calculations. For $k = 4$, the are six permutations (including the one of Theorem~\ref{even Artin relation thm}), and $2k+4 = 10$. 

  \subsection{Artin relations of odd length}
   \label{Artin relations of odd length}
   
In this subsection, we prove Theorem~\ref{odd Artin relation theorem}, which provides Artin relations of every odd length in Mod(S). In a way, the relations of Theorem~\ref{odd Artin relation theorem} are generalizations of the famous Artin relation of length three (also called the braid relation), to all odd lengths.
   
\begin{proof}[Proof of Theorem~\ref{odd Artin relation theorem}]
In this proof, we again cancel only from the left. This allows us to pick up where we left when studying the next Artin relation. Right cancellations are intentionally ignored. When $k = 1$, there are two curves $a_1$ and $b_1$ with $i(a_1,b_1) =1$. Hence, $x = A_1$ and $y = B_1$, and by Fact~\ref{braid reln}, $xyx = yxy$. Consequently, $prod(x,y;\ell) = prod(y,x;\ell)$ for all positive integers $\ell$ that are multiples of $3$. Conversely, suppose $\ell \not \equiv 0 \mathit{(mod3)}$. If $prod(x,y;\ell) = prod(y,x;\ell)$, then $prod(x,y;r) = prod(y,x;r)$ for some $r \in \left\{ 1,2 \right\}$. Since $i(a_1,b_1) = 1$, $x \neq y$. Moreover, if $xy = yx$, it follows from $xyx = yxy$ that $x = y$, which is a contradiction. This proves the theorem for $k = 1$. So, we henceforth assume that $k \geq 2$.
 \medskip
 
We remark that when $B_{k-2}$ or $B_{k-3}$ are included in any of the computations below, it should be assumed that $k > 3$. The inclusion of such terms in the rather complicated calculations is intended to help the reader follow the proof. Although not included, the calculations for $k \in \{2,3\}$ follow along the same lines of the ones shown (for $k > 3$). In fact, the cases $k \in \{2,3\}$ are much easier because there are less terms involved.
\begin{eqnarray*}
 xy = yx \Leftrightarrow A_1 \cdots A_{k-1}A_kB_1 \cdots B_k &=& B_1 \cdots B_kA_1 \cdots A_{k-1}A_k  \\
         \Leftrightarrow A_1 \cdots A_{k-1}A_kB_1 \cdots B_k &=& A_1 \cdots A_{k-1}B_1 \cdots B_kA_k \\
                          \Leftrightarrow A_kB_1 \cdots B_k  &=& B_1A_kB_2 \cdots B_k
\end{eqnarray*}
\noindent Since $[A_k,B_1] \neq 1$, the last equality of RHS does not hold.
\begin{eqnarray*}
 (xy)x = (yx)y \Leftrightarrow A_kB_1 \cdots B_kA_1 \cdots A_k  &=& B_1A_kB_2 \cdots B_kB_1 \cdots B_{k-1}B_k 
\end{eqnarray*}

\noindent Set $\delta_1 = A_kB_1 \cdots B_kA_1 \cdots A_k$. Then
\begin{eqnarray*}
(xy)x = (yx)y \Leftrightarrow \delta_1 &=& B_1A_kB_2 \cdots B_{k-1}B_1 \cdots B_{k-2}B_kB_{k-1}B_k\\
                                       &=& B_1A_kB_2 \cdots B_{k-2}B_1 \cdots B_{k-1}B_{k-2}B_kB_{k-1}B_k\\
& \vdots \\
                                       &=& B_1A_kB_2B_1B_3B_2 \cdots B_iB_{i-1}B_{i+1}B_i \cdots \\
& &  B_{k-2}B_{k-3}B_{k-1}B_{k-2}(B_kB_{k-1}B_k)\\
                                       &=& B_1A_kB_2B_1B_3B_2 \cdots B_iB_{i-1}B_{i+1}B_i \cdots \\ 
& &B_{k-2}B_{k-3}(B_{k-1}B_{k-2}B_{k-1})B_kB_{k-1}\\
                                       &=& B_1A_kB_2B_1B_3B_2 \cdots B_iB_{i-1}B_{i+1}B_i \cdots  \\ 
& & (B_{k-2}B_{k-3}B_{k-2})B_{k-1}B_{k-2}B_kB_{k-1}\\
& \vdots \\
                                       &=& (B_1A_kB_1)B_2B_1B_3B_2B_4B_3 \cdots B_iB_{i-1}B_{i+1}B_i \cdots \\ 
& & B_{k-3}B_{k-2}B_{k-3}B_{k-1}B_{k-2}B_kB_{k-1}\\
                                       &=& A_kB_1A_kB_2B_1B_3B_2B_4B_3 \cdots B_iB_{i-1}B_{i+1}B_i \cdots \\ 
& & B_{k-3}B_{k-2}B_{k-3}B_{k-1}B_{k-2}B_kB_{k-1}\\
& \vdots \\
                                       &=& A_kB_1B_2B_3 \cdots B_kA_kB_1B_2B_3 \cdots B_{k-1}
\end{eqnarray*}
\noindent In order to get the last expression above, we shifted the second $A_k$ to the right as much as possible, and the $B_i$'s to the left as much as possible. Similar shifts occur in future computations. 
\begin{eqnarray*}
(xy)x = (yx)y \Leftrightarrow A_kB_1 \cdots B_kA_1 \cdots A_k &=& A_kB_1 \cdots B_kA_kB_1 \cdots B_{k-1} \\
                               \Leftrightarrow A_1 \cdots A_k &=& A_kB_1 \cdots B_{k-1}
\end{eqnarray*}
\noindent Acting with the products $A_1 \cdots A_k$ and $A_kB_1 \cdots B_{k-1}$ on $a_1$ yields distinct curves. Consequently, the two products are distinct. 
\begin{eqnarray*}
 (xy)^2 = (yx)^2 \Leftrightarrow A_1 \cdots A_kB_1 \cdots B_k &=& A_kB_1 \cdots B_{k-1}A_1 \cdots A_{k-2}A_{k-1}A_k\\
   \Leftrightarrow A_1 \cdots A_{k-2}A_{k-1}A_kB_1 \cdots B_k &=& A_1 \cdots A_{k-2}A_kB_1 \cdots B_{k-1}A_{k-1}A_k\\
   \Leftrightarrow A_{k-1}A_kB_1 \cdots B_k &=& A_kA_{k-1}B_1A_kB_2 \cdots B_{k-1}
\end{eqnarray*}
Since $A_{k-1}A_kB_1 \cdots B_k(a_{k-1}) = a_k \neq b_1 = A_kA_{k-1}B_1A_kB_2 \cdots B_{k-1}(a_{k-1})$, the two expressions are distinct.
\begin{eqnarray*}
 (xy)^2x &=& (yx)^2y \Leftrightarrow \\
 A_{k-1}A_kB_1 \cdots B_kA_1 \cdots A_k &=& A_kA_{k-1}B_1A_kB_2 \cdots B_{k-1}B_1 \cdots B_k
\end{eqnarray*}

\noindent Set $\delta_2 = A_{k-1}A_kB_1 \cdots B_kA_1 \cdots A_k$. Then $(xy)^2x = (yx)^2y \Leftrightarrow$

\begin{eqnarray*}  
 \delta_2 &=& A_kA_{k-1}B_1A_kB_2 \cdots B_{k-2}B_1 \cdots B_{k-1}B_{k-2}B_{k-1}B_k\\
          &=& A_kA_{k-1}B_1A_kB_2 \cdots B_{k-3}B_1 \cdots B_{k-2}B_{k-3}B_{k-1}B_{k-2}B_{k-1}B_k\\
& \vdots \\
          &=& A_kA_{k-1}B_1A_kB_2B_1B_3B_2 \cdots B_iB_{i-1}B_{i+1}B_i \cdots \\ 
          & & B_{k-2}B_{k-3}(B_{k-1}B_{k-2}B_{k-1})B_k \\
          &=& A_kA_{k-1}B_1A_kB_2B_1B_3B_2 \cdots B_iB_{i-1}B_{i+1}B_i \cdots \\ 
          & & (B_{k-2}B_{k-3}B_{k-2})B_{k-1}B_{k-2}B_k\\
          &=& A_kA_{k-1}B_1A_kB_2B_1B_3B_2 \cdots B_iB_{i-1}B_{i+1}B_i \cdots \\ 
          & & B_{k-3}B_{k-2}B_{k-3}B_{k-1}B_{k-2}B_k\\
& \vdots \\
          &=& (A_kA_{k-1}A_k)B_1A_kB_2B_1B_3B_2 \cdots B_iB_{i-1}B_{i+1}B_i \cdots \\ 
          & & B_{k-3}B_{k-2}B_{k-3}B_{k-1}B_{k-2}B_k\\
          &=& A_{k-1}A_kA_{k-1}B_1A_kB_2B_1B_3B_2 \cdots B_iB_{i-1}B_{i+1}B_i \cdots \\ 
          & & B_{k-3}B_{k-2}B_{k-3}B_{k-1}B_{k-2}B_k\\
          &=& A_{k-1}A_kB_1 \cdots B_kA_{k-1}A_kB_1 \cdots B_{k-2}
\end{eqnarray*}
\begin{eqnarray*}
                                                (xy)^2x &=& (yx)^2y \\
 \Leftrightarrow A_{k-1}A_kB_1 \cdots B_kA_1 \cdots A_k &=& A_{k-1}A_kB_1 \cdots B_kA_{k-1}A_kB_1 \cdots B_{k-2} \\
                         \Leftrightarrow A_1 \cdots A_k &=& A_{k-1}A_kB_1 \cdots B_{k-2}
\end{eqnarray*}
Since $A_1 \cdots A_k(a_1) = a_2 \neq a_1 = A_{k-1}A_kB_1 \cdots B_{k-2}(a_1)$, the two expressions are different.
\noindent Based on the equivalences above, one suspects that equivalent expressions for $(xy)^m = (yx)^m$ and $(xy)^mx = (yx)^my$ can be given in general. We make the following claims:

\begin{clm} \label{claim1o}
Let $k$ and $m$ be a positive integers such that $k \geq 2$ and $m < k$. Then, for all $m$, $(xy)^m = (yx)^m \Leftrightarrow$ 
\begin{eqnarray*}
 A_{k-m+1} \cdots A_kB_1 \cdots B_k &\stackrel{(1)}{=}& A_{k-m+2}A_{k-m+1}A_{k-m+3}A_{k-m+2}A_{k-m+4} A_{k-m+3} \\ 
                                    & & A_{k-m+5}A_{k-m+4} \cdots A_{k-1}B_1A_kB_2 \cdots B_{k-m+1}
\end{eqnarray*}
\end{clm}

\begin{clm} \label{claim2o}
Let $k$ and $m$ be a positive integers such that $k \geq 2$ and $m \leq k$. Then, for all $m$,
\begin{center}
 $(xy)^mx = (yx)^my \Leftrightarrow A_1 \cdots A_k \stackrel{(2)}{=} A_{k-m+1} \cdots A_kB_1 \cdots B_{k-m}$ 
\end{center}
\end{clm}

In the claims above, $k$ represents the number of $A_i$'s in $x$ ($k$ is also equal to the number of $B_i$'s in $y$). If $\ell$ represents the lengths of the Artin relations considered in Claims ~\ref{claim1o} and ~\ref{claim2o}, then
\[
 m =
  \left\{
   \begin{array}{ll}
    \ell/2 & \mbox{when $\ell$ is even} \\
     \medskip
    \frac{\ell - 1}{2} & \mbox{when $\ell$ is odd} \\
   \end{array}
  \right.
\]
 
\noindent \textbf{Proof of Claims~\ref{claim1o} and~\ref{claim2o}} - 
We proceed by induction on $m$. Suppose $(xy)^m = (yx)^m \Leftrightarrow (1)$ holds. First, we prove $(xy)^mx = (yx)^my \Leftrightarrow (2)$ is true, then we show $(xy)^{m+1} = (yx)^{m+1} \Leftrightarrow (1)$ with $m$ replaced with $m+1$. Fix an arbitrary $m$ with $1 \leq m \leq k-2$, and assume Claim~\ref{claim1o} is true for that $m$.
\begin{eqnarray*}
                                         (xy)^mx &=& (yx)^my \Leftrightarrow \\
A_{k-m+1} \cdots A_kB_1 \cdots B_kA_1 \cdots A_k &=& A_{k-m+2}A_{k-m+1}A_{k-m+3}A_{k-m+2}A_{k-m+4} \\
                                                 & & A_{k-m+3}A_{k-m+5}A_{k-m+4} \cdots A_{k-1}B_1A_kB_2\\ 
                                                 & & \cdots B_{k-m+1}B_1 \cdots B_k
\end{eqnarray*}
\noindent Set $\delta_3 = A_{k-m+1} \cdots A_kB_1 \cdots B_kA_1 \cdots A_k$. Also, set $s = k-m$, and note that $s > 0$ since $m < k$.  Then $(xy)^mx = (yx)^my \Leftrightarrow$
\begin{eqnarray*}
A_{s+1} \cdots A_kB_1 \cdots B_kA_1 \cdots A_k &=& A_{s+2}A_{s+1}A_{s+3}A_{s+2}A_{s+4}A_{s+3} \\
                                               & & A_{s+5}A_{s+4} \cdots A_{k-1}B_1A_kB_2 \cdots \\ 
                                               & & B_{s+1}B_1 \cdots B_k \\
\end{eqnarray*}
\begin{eqnarray*}                                               
                      \Leftrightarrow \delta_3 &=& A_{s+2}A_{s+1}A_{s+3}A_{s+2}A_{s+4}A_{s+3}A_{s+5}A_{s+4} \cdots A_{k-1}B_1A_kB_2 \cdots \\ 
                                               & & B_sB_1 \cdots B_{s-1}B_{s+1}B_sB_{s+1} \cdots B_k\\
                                               &=&  A_{s+2}A_{s+1}A_{s+3}A_{s+2}A_{s+4}A_{s+3}A_{s+5}A_{s+4} \cdots A_{k-1}B_1A_kB_2 \cdots \\ 
                                               & & B_{s-1}B_1 \cdots B_{s-2}B_sB_{s-1}B_{s+1}B_sB_{s+1} \cdots B_k \\
& \vdots 
\end{eqnarray*}
\smallskip

\begin{eqnarray*}
                                               &=& A_{s+2}A_{s+1}A_{s+3}A_{s+2}A_{s+4}A_{s+3}A_{s+5}A_{s+4} \cdots A_{k-1}B_1A_kB_2B_1B_3B_2B_4B_3 \\ 
                                               & & B_5B_4 \cdots B_sB_{s-1}(B_{s+1}B_sB_{s+1})B_{s+2}B_{s+3} \cdots B_k \\
                                               &=& A_{s+2}A_{s+1}A_{s+3}A_{s+2}A_{s+4}A_{s+3}A_{s+5}A_{s+4} \cdots A_{k-1}B_1A_kB_2B_1B_3B_2B_4B_3 \\ 
                                               & & B_5B_4 \cdots (B_sB_{s-1}B_s)B_{s+1}B_sB_{s+2}B_{s+3} \cdots B_k \\
                                               &=& A_{s+2}A_{s+1}A_{s+3}A_{s+2}A_{s+4}A_{s+3}A_{s+5}A_{s+4} \cdots A_{k-1}B_1A_kB_2B_1B_3B_2B_4B_3 \\ 
                                               & & B_5B_4 \cdots (B_{s-1}B_{s-2}B_{s-1})B_sB_{s-1}B_{s+1}B_sB_{s+2}B_{s+3} \cdots B_k \\
& \vdots \\      
                                               &=& A_{s+2}A_{s+1}A_{s+3}A_{s+2}A_{s+4}A_{s+3}A_{s+5}A_{s+4} \cdots A_{k-1}B_1A_k(B_2B_1B_2)B_3B_2 \\ 
                                               & & B_4B_3B_5B_4 \cdots B_sB_{s-1}B_{s+1}B_sB_{s+2}B_{s+3} \cdots B_k \\
                                               &=& A_{s+2}A_{s+1}A_{s+3}A_{s+2}A_{s+4}A_{s+3}A_{s+5}A_{s+4} \cdots A_{k-1}(B_1A_kB_1)B_2B_1B_3B_2 \\ 
                                               & & B_4B_3B_5B_4 \cdots B_sB_{s-1}B_{s+1}B_sB_{s+2}B_{s+3} \cdots B_k \\
                                               &=& A_{s+2}A_{s+1}A_{s+3}A_{s+2}A_{s+4}A_{s+3}A_{s+5}A_{s+4} \cdots A_{k-1}A_{k-2}(A_kA_{k-1}A_k) \\ 
                                               & & B_1A_kB_2B_1B_3B_2B_4B_3B_5B_4 \cdots B_sB_{s-1}B_{s+1}B_sB_{s+2}B_{s+3} \cdots B_k \\              & \vdots \\
                                               &=& (A_{s+2}A_{s+1}A_{s+2})A_{s+3}A_{s+2}A_{s+4}A_{s+3}A_{s+5}A_{s+4}A_{s+6}A_{s+5} \cdots A_{k-1}A_{k-2} \\
                                               & & A_kA_{k-1}B_1A_kB_2B_1B_3B_2B_4B_3B_5B_4 \cdots B_sB_{s-1}B_{s+1}B_sB_{s+2}B_{s+3} \cdots B_k \\
                                               &=& A_{s+1}A_{s+2}A_{s+1}A_{s+3}A_{s+2}A_{s+4}A_{s+3}A_{s+5}A_{s+4}A_{s+6}A_{s+5} \cdots A_{k-1}A_{k-2}A_k \\ 
                                               & & A_{k-1}B_1A_kB_2B_1B_3B_2B_4B_3B_5B_4 \cdots B_sB_{s-1}B_{s+1}B_sB_{s+2}B_{s+3} \cdots B_k \\
                                               &=& A_{s+1}A_{s+2}A_{s+3} \cdots A_kB_1 \cdots B_kA_{s+1}A_{s+2}A_{s+3} \cdots A_kB_1 \cdots B_s
\end{eqnarray*}
\noindent Thus, $(xy)^mx = (yx)^my \Leftrightarrow$
\begin{eqnarray*}
 A_{s+1}A_{s+2}A_{s+3} \cdots A_kB_1 \cdots B_kA_1 \cdots A_k &=& A_{s+1}A_{s+2}A_{s+3} \cdots A_kB_1 \cdots B_k  \\
                               \Leftrightarrow A_1 \cdots A_k &=& A_{s+1}A_{s+2}A_{s+3} \cdots A_kB_1 \cdots B_s \\   
                               \Leftrightarrow A_1 \cdots A_k &=& A_{k-m+1}A_{k-m+2}A_{k-m+3} \cdots A_k \\ 
                                                              & & \cdots B_1 \cdots B_{k-m}                                                       \end{eqnarray*}
Note that this is Claim~\ref{claim2o} for $m$. We now show that Claim~\ref{claim2o} for $m$ implies Claim~\ref{claim1o} for $m+1$. Again, we set $s = k-m$. 
\begin{eqnarray*}
                                                      (xy)^{m+1} &=& (yx)^{m+1} \\
                    \Leftrightarrow A_1 \cdots A_kB_1 \cdots B_k &=& A_{s+1}A_{s+2}A_{s+3} \cdots A_kB_1 \cdots B_sA_1 \cdots A_k\\
\end{eqnarray*}
\begin{eqnarray*}
  \Leftrightarrow A_1 \cdots A_{s-1}A_s \cdots A_kB_1 \cdots B_k &=& A_1 \cdots A_{s-1}A_{s+1}A_{s+2}A_{s+3} \cdots A_kB_1 \cdots \\ 
                                                                 & & B_sA_sA_{s+1}A_{s+2} \cdots A_k\\
                    \Leftrightarrow A_s \cdots A_kB_1 \cdots B_k &=& A_{s+1}A_{s+2}A_{s+3} \cdots A_kB_1 \cdots B_sA_sA_{s+1} \\ 
                                                                 & & A_{s+2} \cdots A_k\\
                    \Leftrightarrow A_s \cdots A_kB_1 \cdots B_k &=& A_{s+1}A_sA_{s+2}A_{s+1} A_{s+3}A_{s+2}A_{s+4}A_{s+3}A_{s+5} \\ 
                                                                 & & A_{s+4} \cdots A_{k-1}A_{k-2}A_kA_{k-1} \\ 
                                                                 & & B_1A_kB_2 \cdots B_s\\
                \Leftrightarrow A_{k-m} \cdots A_kB_1 \cdots B_k &=& A_{k-m+1}A_sA_{k-m+2}A_{s+1} A_{k-m+3}A_{k-m+2} \\ 
                                                                 & & A_{k-m+4}A_{k-m+3}A_{k-m+5}A_{k-m+4} \cdots A_{k-1} \\ 
                                                                 & & A_{k-2}A_kA_{k-1}B_1A_kB_2 \cdots B_{k-m}\\
\end{eqnarray*} 
Therefore, Claims~\ref{claim1o} and~\ref{claim2o} hold for all $m < k$, by induction. It remains to show that Claim~\ref{claim2o} is true when $m = k$. But this is obvious since, in this case, $(xy)^mx = (yx)^my \Leftrightarrow A_1 \cdots A_k = A_1 \cdots A_k$. 
 \medskip 

Now we prove Theorem~\ref{odd Artin relation theorem}, which is an easy consequence of the two claims. For the sufficient condition, suppose $n = t(2k+1)$, $t \in \mathbb{N}$. It suffices to show $\ell = 2k+1 \Rightarrow prod(x,y;\ell) = prod(y,x;\ell)$. Clearly, $m =k$ when $\ell = 2k+1$. Therefore, it follows by, Claim~\ref{claim2o}, that 
\[
  prod(x,y;\ell) = prod(y,x;\ell) \Leftrightarrow A_1 \cdots A_k = A_{k-k+1} \cdots A_kB_1 \cdots B_{k-k}
\]
\[
  \hspace{2.3cm} \Leftrightarrow A_1 \cdots A_k = A_1 \cdots A_k
\] 

For the necessary condition, assume that $2k+1 \nmid \ell$. Then, $prod(x,y;\ell)= prod(y,x;\ell)$ if and only if $prod(x,y;r) = prod(y,x;r)$ for some $r \in \left\{ 1,2,\cdots,2k \right\}$, where $\ell \equiv r \mathit{(mod(2k+1))}$.
 \smallskip
 
If $r$ is even, then by Claim~\ref{claim1o}, $prod(x,y;r) = prod(y,x;r) \Leftrightarrow$
\begin{eqnarray*}
  A_{k-q+1} \cdots A_kB_1 \cdots B_k &=& A_{k-q+2}A_{k-q+1}A_{k-q+3}A_{k-q+2}A_{k-q+4}A_{k-q+3}A_{k-q+5} \\ 
                                     & & A_{k-q+4} \cdots A_{k-1}B_1A_kB_2 \cdots B_{k-q+1}
\end{eqnarray*}
\noindent where $q = \frac{r}{2}$. Note that since $r \leq 2k$, this equation holds for $q \leq k$, $k \geq 2$.
\begin{eqnarray*}
  LHS(a_{k-q+1}) &=& A_{k-q+1}A_{k-q+2}(a_{k-q+1}) \\
                 &=& a_{k-q+2} 
\end{eqnarray*}
\begin{eqnarray*}
  RHS(a_{k-q+1}) &=&    A_{k-q+2}A_{k-q+1}A_{k-q+3}A_{k-q+2}(a_{k-q+1}) \\
                 &=&    A_{k-q+2}A_{k-q+3}A_{k-q+1}A_{k-q+2}(a_{k-q+1}) \\
                 &=&    A_{k-q+2}A_{k-q+3}(a_{k-q+2})\\
                 &=&    a_{k-q+3}\\
                 &\neq& a_{k-q+2}
\end{eqnarray*}
Hence, LHS $\neq$ RHS.
 \smallskip
 
If $r$ is odd, then $prod(x,y;r) = prod(y,x;r) \Leftrightarrow A_1 \cdots A_k = A_{k-p+1} \cdots A_kB_1 \cdots B_{k-p} $, where $p = \frac{r-1}{2}$. But then, provided that $k-p+1 > 1$, we have $LHS(a_1) = A_1A_2(a_1) = a_2$, while $RHS(a_1) = a_1 \neq a_2$. Hence, LHS $\neq$ RHS.
  \medskip
  
It remains to show that $k-p+1 > 1$. Indeed, $p \geq k \Rightarrow \frac{r-1}{2} \geq k \Rightarrow r \geq 2k+1$, contradicting $r \in \left\{ 1,   \cdots, 2k \right\}$. 
\end{proof}

\section{Artin relations from dihedral foldings}
\label{sec:7}
   
   In this section, we prove Theorems~\ref{bdrelthm1} and~\ref{bdrelthm2}, which provide more Artin relations in the mapping class group. As a matter of fact, these theorems give explicit elements $x$ and $y$ in Mod(S) that generate Artin groups of types $I_2(k)$ and $I_2(2k-2)$ respectively. In the proofs of Theorems~\ref{bdrelthm1} and~\ref{bdrelthm2}, we invoke LCM-homomorphisms (described in subsection~\ref{sec:3}) induced by the dihedral foldings $A_{k-1} \rightarrow I_2(k)$ and $D_k \rightarrow I_2(2k-2)$ respectively. The induced embeddings between the corresponding Artin groups provide $k$ and $2k-2$ Artin relations in $\mathcal{A}(A_{k-1})$ and $\mathcal{A}(D_k)$ respectively. Since the Artin groups of types $A_{k-1}$ and $D_k$ inject into the corresponding mapping class groups via the geometric homomorphism, we obtain Artin relations of length $k$ and $2k-2$ in $Mod(S_{\Gamma})$, $\Gamma = A_{k-1}, D_k$. Here, the discussion follows subsection~\ref{sec:3}, including notation and terminology.
   \medskip
   
\begin{proof}[Proof of Theorem~\ref{bdrelthm1}]
We only prove the case when $k$ is even. The odd case is proved similarly. Assume $k$ is an even integer greater than $3$. Consider the Coxeter graphs $A_{k-1}$ and $I_2(k)$, and label their vertices by the sets $P = \{ s_1,s_2,\cdots,s_{k-1} \}$ and $Q = \{ s,t \}$ respectively. Partition $P$ into $K_s = \{s_1,s_3,\cdots,s_{k-3},s_{k-1} \}$ and $K_t = \{s_2,s_4,\cdots,s_{k-4},s_{k-2} \}$. By Corollary~\ref{corollary to cr3}, the dihedral folding $f : A_{k-1} \rightarrow I_2(k)$ such that $f(K_s) = s$ and $f(K_t) = t$ induces the LCM-homomorphism 
\begin{center}
 $\phi^f : \mathcal{A}^+(I_2(k)) \rightarrow \mathcal{A}^+(A_{k-1})$\\
  \medskip
 \hspace{1.5cm} $s \mapsto \Delta_{f^{-1}(s)}$\\
  \medskip
 \hspace{1.5cm} $t \mapsto \Delta_{f^{-1}(t)}$\\ 
\end{center}
\noindent By Lemma \ref{lcm lem}, $\Delta_{f^{-1}(s)} = s_1 s_3 \cdots s_{k-3} s_{k-1}$ and $\Delta_{f^{-1}(t)} = s_2 s_4 \cdots s_{k-4} s_{k-2}$. Corollary~\ref{corollary to cr3} implies that $\phi^f$ is injective. By Theorem \ref{cr2}, $\phi^f$ induces an injective homomorphism $\phi$ between the corresponding Artin groups. The curve graph associated to the $a_i$ is isomorphic to $A_{k-1}$. Consequently, the geometric homomorphism 
\begin{center}
 $g : \mathcal{A}(A_{k-1}) \rightarrow Mod(S_{A_{k-1}})$\\
  \medskip
\hspace{-0.1cm}$s_i \mapsto T_i$\\
\end{center}
\noindent is injective by Theorem~\ref{pv}. As such, the composition $g \circ \phi$ gives a monomorphism of $\mathcal{A}(I_2(k))$ into $Mod(S_{A_{k-1}})$. 
 \medskip
 
Since $x$ and $y$ generate $\mathcal{A}(I_2(k))$, $prod(x,y;k) = prod(y,x;k)$. From this equality, it follows immediately that $prod(x,y;pk) = prod(y,x;pk)$ for all positive integers $p$. This proves the sufficient condition of the last statement in Theorem~\ref{bdrelthm1}. For the necessary condition, assume $k \nmid n$ and $prod(x,y;n) = prod(y,x;n)$. Then $prod(x,y;r) = prod(y,x;r)$ for some $r \in \left\{ 1, \cdots, k-1 \right\}$. But since $g \circ \phi$ is injective, this would mean that $prod(s,t;r) = prod(t,s;r)$ in $\mathcal{A}(I_2(k))$, which is a contradiction. 
\end{proof}

\begin{figure}
	\centering
		\includegraphics[width=0.70\textwidth]{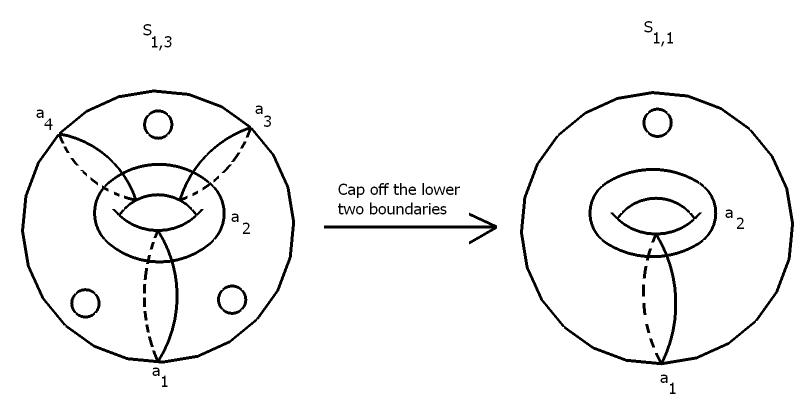}
	\caption{{\small If $x = T_1 T_3 T_4$ and $y = T_2$, it follows from Theorem~\ref{bdrelthm2} that $(xy)^3 = (yx)^3$ in $Mod(S_{1,3})$. By capping off the two boundary components, one gets the relation $(T_1^3 T_2)^3 = (T_2 T_1^3)^3$ in $Mod(S_{1,1})$.}}
	\label{T_1^3 and T_2 satisfy 6-Artin reln}
\end{figure}

\begin{proof}[Proof of Theorem~\ref{bdrelthm2}]
Again, we only prove the even case. The odd one is proved the same way. Label the vertices of $D_k$ and $I_2(2k-2)$ by $P = \{s_1, s_2,\cdots, s_k \}$ and $Q = \{ s,t \}$ respectively. Partition $P$ into $K_s = \{s_1, s_3,\cdots, s_{k-1}, s_k \}$ and $K_t = \{s_2, s_4,\cdots, s_{k-2} \}$. Both $K_s$ and $K_t$ consist of pairwise commuting generators of $\mathcal{A}^+(D_k)$. By Lemma \ref{lcm lem}, the least common multiple of each of these sets is the product of its elements (in any order). The dihedral folding $f: D_k \rightarrow I_2(2k-2)$ induces the LCM-homomorphism $\phi^f : \mathcal{A}^+(I_2(2k-2)) \rightarrow \mathcal{A}^+(D_k)$, which maps
\begin{center}
 $s \mapsto \Delta_{f^{-1}(s)} = s_1 s_3 \cdots s_{k-1} s_k $\\
  \medskip
 $t \mapsto \Delta_{f^{-1}(t)} = s_2 s_4 \cdots s_{k-2}$\\ 
\end{center}
\noindent Since $\phi^f$ is injective, the induced map, $\phi$, on the corresponding Artin group is injective as well. By post-composing with the geometric homomorphism $g: \mathcal{A}(D_k) \rightarrow Mod(S_{D_k})$, one gets an embedding of $I_2(2k-2)$ into $Mod(S_{\frac{k-2}{2},3})$. This produces a subgroup of $Mod(S_{\frac{k-2}{2},3})$ which is isomorphic to the Artin group $A(I_2(2k-2))$, and is generated by $x$ and $y$. Since $prod(s,t;2k-2) = prod(t,s;2k-2)$, it follows that $prod(x,y;2k-2) = prod(y,x;2k-2)$. The rest of the proof follows as in Theorem \ref{bdrelthm1}.
\end{proof}

Finally, the results in this paper also imply other relations in the mapping class group. For example, it was pointed to us by Dan Margalit that when $i(a,b) = 1$, $T_a^3$ and $T_b$ satisfy an Artin relation of length $6$. In the following corollary, we show that this fact follows from Theorem~\ref{bdrelthm2}

\begin{cor}[Corollary to Theorem~\ref{bdrelthm2}]\label{corollary to bdrelthm2}
Let $a_1$ and $a_2$ be isotopy classes of simple closed curves in $S$. If $i(a_1,a_2) = 1$, then $T_1^3$ and $T_2$ satisfy an Artin relation of length $6$.
\end{cor}
\begin{proof}
Let $F$ be a regular neighborhood of $a_1 \cup a_2$ so that $F$ is homeomorphic to $S_{1,1}$. In $F$, consider three parallel copies of $a_1$ denoted by $a_1$, $a_3$, and $a_4$. Now remove two open disks from $F$ to obtain $S_{1,3}$ and curves $a_1, a_2, a_3$, and $a_4$ as in Figure~\ref{T_1^3 and T_2 satisfy 6-Artin reln}.
 \medskip

When $k = 4$, Theorem~\ref{bdrelthm2} implies that $x = T_1 T_3 T_4$ and $y = T_2$ satisfy $xyxyxy = yxyxyx$ in $Mod(S_{D_4}) = Mod(S_{1,3})$. By Theorem~\ref{pv}, the subgroup $G$ of $Mod(S_{1,3})$ generated by $T_1, T_2, T_3$, and $T_4$ is isomorphic to $\mathcal{A}(D_4)$. Now, reverse the process and cap off the two boundary components to recover $F \approx S_{1,1}$. There is a homomorphism $G \rightarrow Mod(S_{1,1})$ defined by $T_2 \mapsto T_2$ and $T_j \mapsto T_1$ for $j = 1,3,4$. The image of $xyxyxy = yxyxyx$ under this homomorphism is $(T_1^3 T_2)^3 = (T_2 T_1^3)^3$ in $Mod(S_{1,1})$. Of course, the same relation is true in Mod(S).
\end{proof}

\noindent \textbf{Acknowledgements.} I am very grateful to my advisor, Sergio Fenley, for his revisions of this article and for the many stimulating conversations. I am thankful to Dan Margalit for his valuable suggestions, and to Chris Leininger for his positive feedback. I am also indebted to the referee for suggesting many improvements to the original submission.


\begin{thebibliography}{99}
{\small

\bibitem{BS}
Brieskorn, E. and Saito, K.
\emph{Artin-Gruppen und Coxeter-Gruppen},
Invent. Math 17 (1972), 245-271.

\bibitem{Cr}
Crisp, J.
\emph{Injective maps between Artin groups},
Geom. Group Theory Down Under (Canberra 1996) 119–137.

\bibitem{FM}
Farb and Margalit.
\emph{A Primer on Mapping Class Groups}. Available at
http://www.math.utah.edu/~margalit/primer/

\bibitem{H}
Humphreys, J.
\emph{Reflection Groups and Coxeter Groups},
Cambridge University Press, Cambridge, 1990.

\bibitem{PV}
Perron, F. and Vannier, J.P.
\emph{Groupe de monodromie geometrique de singularite simples},
CRAS Paris 315 (1992), 1067-1070.

\bibitem{W1}
Wajnryb, B.
\emph{Artin groups and geometric monodromy}, 
Inventiones Mathematicae 138 (1999), 563-571.

\bibitem{W2}
Wajnryb, B.
\emph{Relations in the mapping class group}, 
in Problems on Mapping Class Groups and Related Topics},
(Volume 74 of Proceedings of symposia in pure mathematics), AMS Bookstore, (2006), 122-128.
\end{thebibliography}
\end{document}